\documentclass[11pt]{article}
\usepackage[margin=1in]{geometry}
\usepackage{bookmark}
\usepackage{graphicx}
\usepackage{float}
\usepackage{amsfonts}
\usepackage{amsmath,amsthm,amssymb}
\usepackage{amsmath,amssymb}
\usepackage{epstopdf}
\usepackage{subfigure}
\usepackage{algorithm, algorithmic}
\usepackage[font=footnotesize]{caption}
\usepackage{color}
\usepackage{wrapfig}
\usepackage{marvosym}
\usepackage{caption}
\usepackage{ragged2e}
\usepackage[scaled]{helvet} 
\usepackage{fourier}

\newcommand{\added}[1]{{#1}}
\newcommand{\C}{\mathbb{C}}

\newcommand{\cS}{\mathcal{S}}
\newcommand{\cP}{\mathcal{P}}
\newcommand{\cR}{\mathcal{R}}

\newcommand{\bF}{{\boldsymbol{F}}}

\newcommand{\bD}{\boldsymbol{D}}

\newcommand{\bJ}{\boldsymbol{J}}
\newcommand{\bP}{\cP}
\newcommand{\br}{\boldsymbol{r}}
\newcommand{\bx}{\boldsymbol{x}}

\newcommand{\balpha}{\boldsymbol{\alpha}}
\newcommand{\bA}{\boldsymbol{A}}
\newcommand{\by}{\boldsymbol{y}}

\newcommand{\be}{\boldsymbol{e}}
\newcommand{\bn}{{\boldsymbol{n}}}

\newcommand{\bG}{\boldsymbol{G}}

\newcommand{\bs}{\boldsymbol{s}}
\newcommand{\bz}{\boldsymbol{z}}
\newcommand{\bw}{\boldsymbol{w}}
\newcommand{\bv}{\boldsymbol{v}}

\newcommand{\Id}{\boldsymbol{Id}}
\newcommand{\bal}{\boldsymbol{\alpha}}

\newcommand{\bvarphi}{\boldsymbol{\varphi}}

\newcommand{\bphi}{\boldsymbol{\phi}}
\newcommand{\bPhi}{\boldsymbol{\Phi}}
\newcommand{\bI}{\boldsymbol{I}}

\newcommand{\supp}{\mathrm{supp}}
\newcommand{\opt}{\mathrm{opt}}
\newcommand{\argmin}{\mathop{\mathrm{argmin}}}

\newtheorem{thm}{Theorem}[section]
\newtheorem{cor}[thm]{Corollary}
\newtheorem{lem}[thm]{Lemma}

\newtheorem{prop}[thm]{Proposition}
\newtheorem{defn}[thm]{Definition}

\numberwithin{equation}{subsection}

\title{A note on practical approximate projection schemes in signal space methods}

\author{Xiaoyi Gu, Deanna Needell and Shenyinying Tu
\thanks{This work was supported by NSF grant DMS-$1045536$, NSF CAREER grant $\#1348721$, and the Alfred P. Sloan Fellowship.}}

\date{\today}

\begin{document}
\maketitle

\begin{abstract}
Compressive sensing (CS) is a new technology which allows the acquisition of signals directly in compressed form, using far fewer measurements than traditional theory dictates.  Recently, many so-called \textit{signal space} methods have been developed to extend this body of work to signals sparse in arbitrary dictionaries rather than orthonormal bases.  In doing so, CS can be utilized in a much broader array of practical settings. 
Often, such approaches often rely on the ability to optimally project a signal onto a small number of dictionary atoms. Such optimal, or even approximate, projections have been difficult to derive theoretically. 
Nonetheless, it has been observed experimentally that conventional CS approaches can be used for such projections, and still provide accurate signal recovery.  In this letter, we summarize the empirical evidence and clearly demonstrate for what signal types certain CS methods may be used as approximate projections.  In addition, we provide theoretical guarantees for such methods for certain sparse signal structures.  Our theoretical results match those observed in experimental studies, and we thus establish both experimentally and theoretically that these CS methods can be used in this context.
\end{abstract}

\section{Introduction}
Compressive sensing (CS) addresses some drawbacks of traditional signal acquisition by suggesting that a signal may be acquired directly in compressed form, and often with far less samples than previously thought.  The key assumption behind this methodology is that signals of interest are \textit{sparse} or nearly sparse.  Typically, this notion can be quantified by assuming that the signal $\bx\in\mathbb{C}^n$ can be represented as $\bx = \bD\balpha$ where $\balpha$ is sparse ($\|\balpha\|_0 := |\supp(\balpha)| = s \ll n$) and $\bD$ is an orthonormal basis called the sparsifying basis.  For example, if $\bD$ is a wavelet basis, most natural images are considered sparse in that basis.  However, in many practical applications, the sparsifying basis is not orthonormal but instead a redundant, highly overcomplete frame.  In this setting, most classical CS approaches break down, and new technologies are needed to address this practical issue.  Recently, both optimization based and greedy methods have begun to be developed; however, many of these approaches rely on the existence of approximately optimal projections to project signals onto sparse subspaces of the frame.  Empirical evidence has suggested that classical CS methods can be used for such projections, but these results have lacked theoretical support.

The mathematical formulation of the CS problem can be described as follows. Denote the signal of interest by $\bx \in \C^n$. The measurement vector $\by \in \C^m$ is obtained using the measurement matrix $\bA \in \C^{m \times n}$ (where $m \ll n$) and written $\by = \bA\bx + \be$, where $\be \in \C^m$ is additive noise.  The sparsity condition on the signal $\bx$ is that $\bx = \bD \bal$, for some coefficient vector $\bal \in \C^d$ with $\|\bal\|_0 \leq k \ll n$, yielding a $k$-sparse representation of the signal $\bx$ with respect to the dictionary $\bD \in \C^{n \times d}$ ($n\leq d$).

In the simplest classical CS setting, the sparsifying basis $\bD$ is just the identity matrix.  In this setting, it is clear that as long as the meausurement matrix $\bA$ is one-to-one on $k$-sparse vectors, then one can recover a $k$-sparse vector $\bx$ from its noiseless measurements $\by = \bA\bx$ by searching for the sparsest vector among all those that match the measurements $\by$.  This 
$\ell_0$-minimization problem of course is NP-hard in general \cite{RefWorks:279}, but motivates the seminal work of Cand\`{e}s, Romberg and Tao \cite{RefWorks:48, RefWorks:47} which shows that a sparse vector $\bx$ can also be recovered using the relaxed $\ell_1$-minimization method, under a slightly stricter assumption on the matrix $\bA$.   
Cand\`es and Tao \cite{RefWorks:48} introduce the Restricted Isometry Property (RIP) which asks that $\bA$ satisfy the following for a constant $\delta_k \in (0,1)$:

\begin{equation}
\label{RIP}
{(1-\delta_k)}\|\bx\|_2^2 \leq \|\bA \bx\|_2^2  \leq {(1+\delta_k)}\|\bx\|_2^2\quad\text{for all $\bx$ with $\|\bx\|_0 \leq k$.}
\end{equation}

Fortunately, it is now well-known that many classes of randomly constructed measurement matrices satisfy this condition.  For example, if the entries of $\bA$ are drawn from a subgaussian distribution and $m\geq Ck/\log (n)$, then $\bA$ satisfies the RIP with high probability \cite{RefWorks:444,RefWorks:285}.  Analogous results hold for subsampled Discrete Fourier Transform (DFT) matrices and others with a fast-multiply \cite{RefWorks:285, rauhut2012restricted}.  
Under the assumption of the RIP, the $\ell_1$-minimization problem is guaranteed to robustly reconstruct a $k$-sparse signal $\bx$ from its noisy measurements $\by=\bA\bx + \be$.
Greedy methods such as OMP~\cite{Paper9, RefWorks:150}, ROMP~\cite{RefWorks:44}, CoSaMP~\cite{NeedeT_CoSaMP}, and IHT~\cite{PaperIHT}, which aim to identify the signal support iteratively, also provide robust reconstruction guarantees in this setting.  We refer the reader to the above references for details about each of the algorithmic guarantees.

\section{Signal Space Methods}

The above mentioned CS methods provide rigorous reconstruction guarantees when the signal of interest $\bx$ is sparse in an orthonormal basis.  Indeed, if we write $\bx=\bD\balpha$ for some sparse coefficient vector $\balpha$, these approaches aim at recovering the representation $\balpha$.  When $\bD$ is orthonormal, this of course translates robustly to the recovery of the actual signal $\bx$.   
However, a vast array of signals in practice are compressible not in an orthonormal basis but in some highly overcomplete dictionary such as an oversampled DFT, Gabor frame, or many of the redundant dictionaries used in signal and image processing.
Due to the effect of redundancy, these classical approaches fail both theoretically and empirically when $\bD$ is no longer orthonormal. 

To address this issue, a recent surge of work has been focused on so-called \textit{signal space} methods.  Two models to capture sparsity in an overcomplete frame have been studied; the synthesis sparsity model asserts that $\bx=\bD\balpha$ for some sparse representation $\balpha$, as discussed above.  The analysis sparsity or \textit{cosparsity} model instead captures the sparsity in the analysis coefficients $\bD^*\bx$ \cite{RefWorks:271,RefWorks:607,foucart15}.  We utilize the synthesis sparsity model in this work, and refer the reader to aforementioned references for a description of some alternatives.

In \cite{RefWorks:23,Paper5}, Davenport et. al. introduce and analyze a greedy method called Signal Space CoSaMP (SSCoSaMP).  SSCoSaMP is an adaptation of the CoSaMP method designed to recover the signal $\bx$ when $\bD$ is an overcomplete frame. Rather than assuming the classical RIP, they instead utilize a generalization of this property called the $D$-RIP, introduced in \cite{RefWorks:60}.   
In contrast to other CS algorithms, SSCoSaMP employs a ``signal-focused" approach, aimed at determining the signal $\bx$ directly rather than its (non-unique) coefficient vector $\balpha$.  The method is described by the pseudo-code given in Algorithm \ref{ccosamp}.  Here and throughout, we denote the range of a matrix $\bD$ by $\cR(\bD)$.

The identification step in SSCoSaMP requires finding the best $k$-sparse representation of a vector $\bz$ in the dictionary $\bD$:
\begin{equation}\label{opt}
\Omega_{\text{opt}} := \argmin_{\Lambda: |\Lambda| = k} \|{\bz - \cP_{\Lambda} \bz}\|_2,
\end{equation}
where here and throughout $\cP_{\Lambda}$ denotes the projection onto the span of the columns of $\bD$ indexed by $\Lambda$.  This problem is itself reminiscent of the classical CS problem; one wishes to recover a sparse representation from an underdetermined linear system.  Thus, such an endeavor in general is itself an NP-hard problem.   
Therefore, we relax this problem and instead ask for a \emph{near-optimal} approximation~\cite{Paper5,giryes2013greedy}, writing $\cS_{\bD}({\bw}, k)$ to denote the $k$-sparse approximation to $\bw$ in $\bD$.   Because of its similarity to the classical CS problem, one might think to use a classical CS approach to solve this sub-problem.  Of course, for typical redundant frames, the RIP will not be in force, and classical results would not suggest accurate recovery. Nonetheless, SSCoSaMP is surprisingly able to accurately recover signals even when $\bD$ is highly overcomplete, when using a classical CS algorithm like OMP, CoSaMP or $\ell_1$-minimization for the near-optimal projection $\cS_{\bD}$.

Even more interesting is that the behavior of SSCoSaMP varies significantly depending on both the approximate projection used, and the structure of the sparse signal $\bx$.  This was originally noted in the experiments of \cite{Paper5}, which we demonstrate again here in Figure \ref{sepvclust} as in \cite{practicalGGKLNT14}.  Here and throughout this paper, the algorithm in parentheses following ``SSCoSaMP" is the algorithm used for the approximate projection $\cS_{\bD}$ in the identify and prune steps.  As is shown, when the support of the signal $\bx$ is clustered together, the CoSaMP method as an approximate projection yields accurate recovery whereas $\ell_1$-minimization and OMP do not perform well at all.  On the other hand, when the signal support has adequate separation, the exact opposite behavior is seen.  This observation led to a recent work which empirically investigated this behavior, cataloging how various approximate projection methods behave for various signal types \cite{practicalGGKLNT14}.

\begin{algorithm}[H]
\caption{Signal-Space CoSaMP (SSCoSaMP)}\label{ccosamp}
\begin{algorithmic}

\STATE \textbf{Input:} $\bA$, $\bD$, $\by$, $k$, stopping criterion 
\STATE \textbf{Initialize:} $\br = \by$, $\bx_0 = 0$, $\ell = 0$, $\Gamma = \varnothing$

\WHILE{not converged}
\STATE
\begin{tabular}{ll}

\textbf{Proxy:} & $\widetilde{\bv} = \bA^*\br$  \\ 
\textbf{Identify:} & $\Omega = \cS_{\bD}(\widetilde{\bv},2k)$ \\
\textbf{Merge:} & $T = \Omega \cup \Gamma$ \\
\textbf{Update:} & $\widetilde{\bw} = \argmin_{z}\|\by - \bA \bz\|_2 \quad   \mathrm{s.t.}  \quad  \bz \in \cR(\bD_T)$	\\
\textbf{Prune:} & $\Gamma = \cS_{\bD}(\widetilde{\bw}, k)$\\
 & $\bx_{\ell+1} = \cP_{\Gamma}\widetilde{\bw}$\\
& $\br = \by - \bA\bx_{\ell + 1}$\\
& $\ell = \ell + 1$\\

\end{tabular}
\ENDWHILE

\STATE \textbf{Output:} $\hat{\bx} = \bx_{\ell}$

\end{algorithmic}
\end{algorithm}

\begin{figure}

\centering
\hspace{-.075\linewidth}
\begin{minipage}{.45\linewidth}
\centering
\begin{tabular}{cc}
\raisebox{25mm}{ {(a)}} &\includegraphics[width=3in]{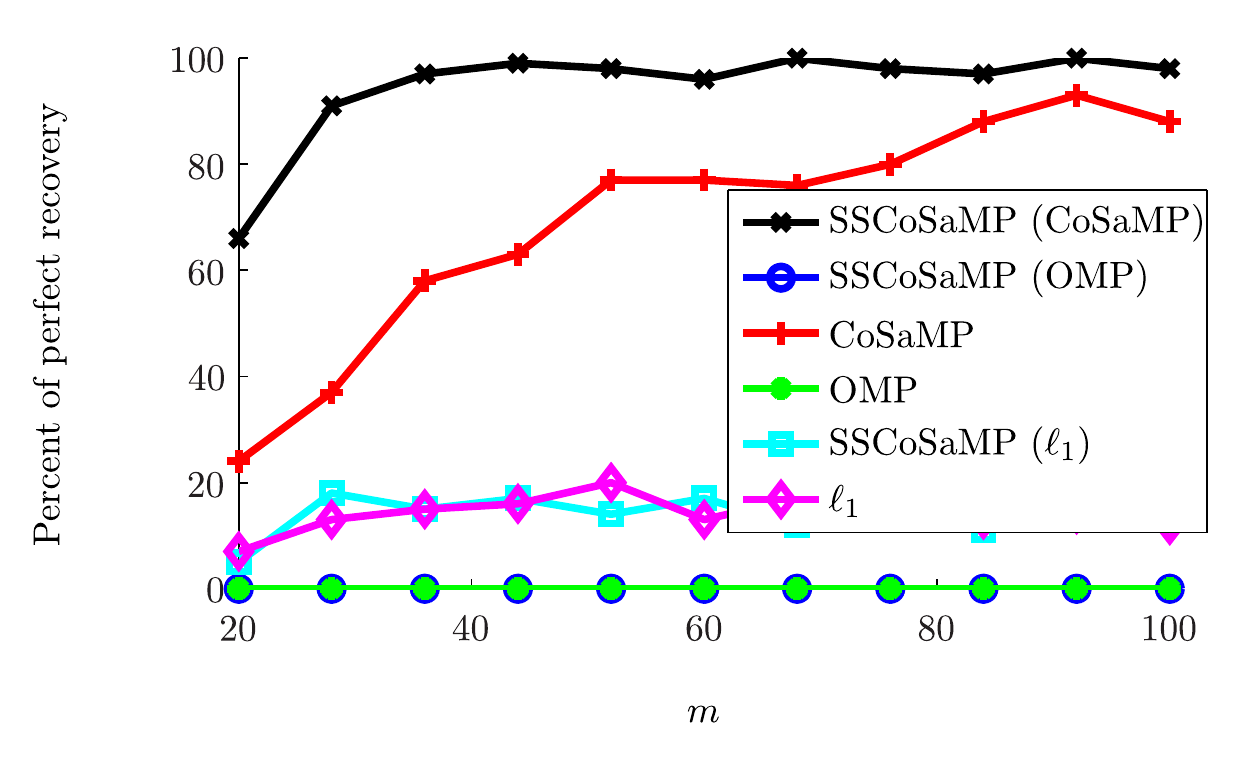}
\end{tabular}
\end{minipage}
\hspace{.05\linewidth}
\begin{minipage}{.45\linewidth}
\centering
\begin{tabular}{cc}
\raisebox{25mm}{ {(b)}} &\includegraphics[width=3in]{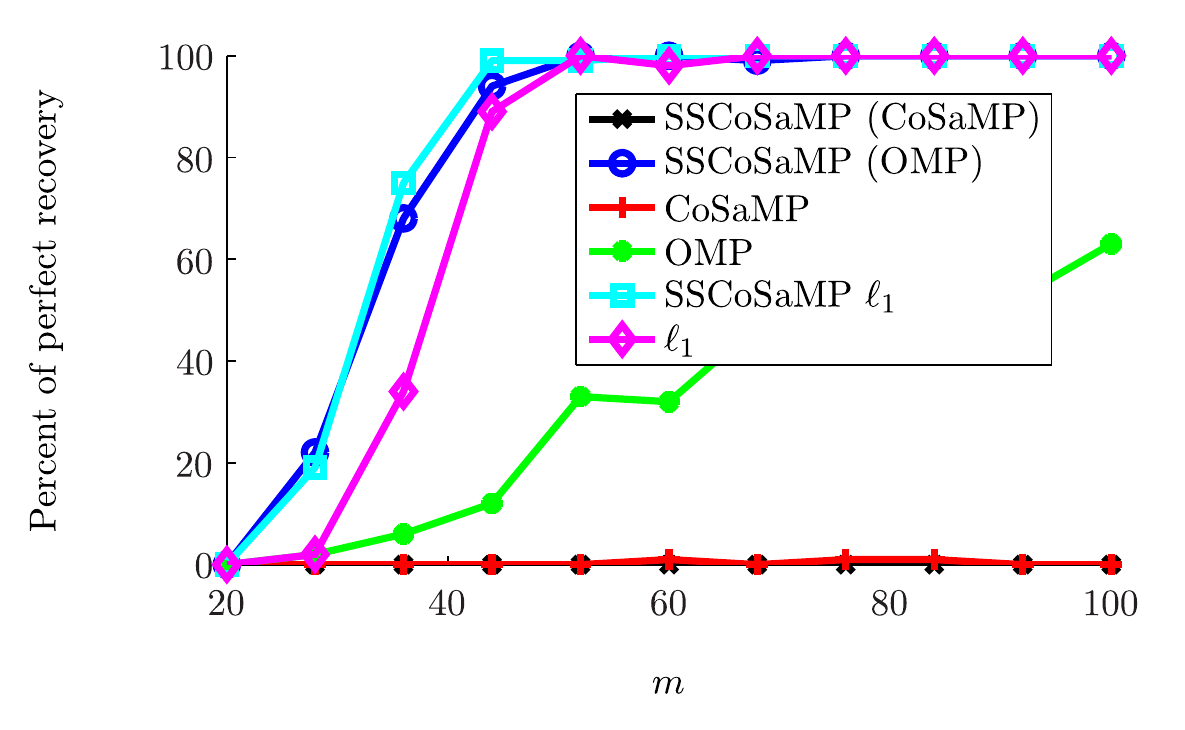}
\end{tabular}
\end{minipage}
   \caption{ From \cite{practicalGGKLNT14}.  Percent of perfect recovery (we define perfect recovery as signal to noise ratio $20\log(\|\bx\|_2/\|\bx-\hat{\bx}\|_2)$ greater than 100db) out of $100$ trials for different SSCoSaMP variants when the nonzero entries in $\balpha$ are clustered together (left) and when the nonzero entries in $\balpha$ are well-separated (right).  Here, $k = 8$, $n = 256$, $d = 1024$, the dictionary $\bD \in \mathbb{C}^{n\times d}$ is a $4\times$ overcomplete DFT and $\bA \in\mathbb{R}^{m\times n}$ is a Gaussian matrix.  
   \label{sepvclust}}
\end{figure}

\subsection{Contribution}
As observed in Figure~\ref{sepvclust} and in more detail in \cite{practicalGGKLNT14}, the accuracy of SSCoSaMP depends on both the approximate projection and the structure of the signal.  In addition, none of the approximate projections simultaneously work well for both signal types (however, see \cite{giryes2013omp,practicalGGKLNT14} for hybrid algorithms that are accurate for wider variety of signal structures).  The discrepancy between the theoretical guarantees for classical CS approaches and the empirical behavior of such methods as approximate projections motivates the work of this note.  Here, we provide an analysis of two of the classical approaches which provably show that for certain signal structures, these methods can be used as accurate approximate projections and thus can be utilized in SSCoSaMP for efficient signal recovery.  We believe our analysis furthers the understanding of the behavior of such methods, which is interesting in its own right as well as in the context of signal space methods.

\section{Analytical Justification}

Recall that \eqref{opt} is NP-hard because it requires examining all $k$ combinations of the columns of $\bD$. A concession one makes is to instead look for a near-optimal projection, as used in Algorithm \ref{ccosamp}. It has been shown \cite{giryes2013greedy} that as long as the near optimal projection is good enough, i.e., 

\begin{equation} \label{eq:eq2}
\|\cP_{\cS_D (x,k)} \bx - \bx\|_2\leq C\|\cP_{\Lambda} \bx - \bx\|_2 \quad\text{and}\quad \|\cP_{\cS_D (x,k)} \bx\|_2 \geq c\|\cP_{\Lambda} \bx \|_2 
\end{equation}
for all $\bx$ and suitable constants $c$ and $C$ (and where $\cP_{\Lambda}$ denotes the optimal projection), then SSCoSaMP provides accurate recovery of the signal.  Although there do exist such projections for well behaved dictionaries $\bD$, for truly redundant dictionaries such projections are not known to exist.  Nonetheless, empirical studies \cite{Paper5} using traditional compressed sensing algorithms for these projections showed that in certain cases SSCoSaMP using these projections still yields accurate recovery.  In this section we build on state of the art results in compressed sensing to provide theoretical justification for this behavior; motivated by Fig. \ref{sepvclust}, we examine the case when SSCoSaMP($\ell_1$) and SSCoSaMP(OMP) are applied to well-separated signals.  We assume here that the dictionary $\bD$ is an overcomplete discrete Fourier matrix (DFT); results for similar dictionaries follow analogously.

\subsection{SSCoSaMP(OMP)}
 Numerical experiments conducted on SSCoSaMP(OMP) suggest that when the sparse vector $\balpha$ in $\bx = \bD \balpha$ is well-separated, SSCoSaMP(OMP) gives accurate recovery of the signal $\bx$.  We thus seek theoretical backing for the use of OMP as an approximate projection within SSCoSaMP in this scenario.  
First, suppose the signal $\bx$ has a $k$-sparse representation in the $n \times d$ matrix $\bPhi$,

\begin{equation} \label{eq:e2}
\bx = \sum_{i\in\Omega_{\opt}} \balpha_{i} \bvarphi_{i}
\end{equation}

 \noindent where $\Omega_{\opt}$ is an index set of size $k$, and $\bvarphi_{i}$ and $\balpha_i$ denote the $i$th column of $\bPhi$ and element of $\balpha$, respectively. Then we can extract the matrix $\bPhi_{\Omega_{\opt}}$  from the dictionary $\bPhi$ whose columns are listed in $\Omega_{\opt}$,

\begin{equation} \label{eq:e3}
\bPhi_{\Omega_{\opt}} = [\bvarphi_{i_1} \, \, \, \bvarphi_{i_2}\,  \cdots \, \bvarphi_{i_k}]_{i_1, \ldots, i_k \in \Omega_{\opt}}.
\end{equation}
 
Then the signal can be expressed as 

\begin{equation} \label{eq:e4}
\bx = \bPhi_{\Omega_{\opt}} \balpha_{\Omega_{\opt}}.
\end{equation}

We will utilize the following established result of 
Cai and Wang \cite{OMPnoise} that improves upon the seminal work of Gilbert and Tropp \cite{Paper9}.  They prove a sufficient condition to accurately recover the signal from contaminated samples when $\bPhi$ is sufficiently incoherent.  Also see these works for a detailed description of OMP.

\begin{prop} \label{minentry} \cite[Prop. 1]{OMPnoise}
Let $\;\Omega_{\opt}$ denote the support set of the signal $\balpha$, set $M := \max_{i \in \Omega_{\opt}^c}\|\bPhi^{\dagger}_{\Omega_{\opt}} \bphi_i\|_1$, let $\br_i$ be the residual vector of OMP in the $i$th iteration, and let $\lambda_{\min}$ denote the minimal eigenvalue of the matrix $\bPhi_{\Omega_{\opt}}^* \bPhi_{\Omega_{\opt}}$.
Suppose $\|\be\|_2 \leq \varepsilon$ and $M < 1$. Then the OMP algorithm (with the stopping rule $\|\br_i\|_2 \leq \varepsilon$) recovers $\Omega_{\opt}$ exactly if all the nonzero coefficients $\balpha_i$ satisfy 

\begin{equation}
|\balpha_i| \geq \frac{2 \varepsilon}{(1- M)\lambda_{\min}}
\end{equation}

\end{prop}

In our setting, $\bD=\bPhi$ is not incoherent, but we want to show that when the support of $\balpha$ is well-separated, then $\max_{i\in \Omega_{\opt}^c} \| \bPhi_{\Omega_{\opt}}^{\dag} \bphi_i\|_1 \leq 1$. From now on, we take $\bD$ to be an $n \times d$ ($n \leq d$) overcomplete DFT matrix (having unit-norm columns), but the technique can be extended to other overcomplete dictionaries.  Recall that if $\bD$ is an $n \times d$ overcomplete DFT dictionary, then the entries of $\bD$ obey (with our normalization):

\begin{equation} \label{eq:e10}
\bD_{jk} = \frac{\omega^{jk}}{\sqrt{n}}_{j = 0, 1, ... n-1, k = 0, 1, ... d-1}, \quad\text{where}\quad \omega = \be^{- \frac{2 \pi i }{d}}.
\end{equation}

We will first define a separation condition that will be needed to guarantee accurate recovery.

\begin{defn}[Well-separated]
We say that $\;\Omega_{\opt}$ is \textnormal{well-separated} when $\bD_{\Omega_{\opt}}^*\bD_{\Omega_{\opt}}$ is strictly diagonally dominant.  Recall that we may define for a $k\times k$ matrix $\bJ$,
\begin{equation}\label{diagdom}
\Delta_i(\bJ) := \bJ_{ii} - \sum_{j \neq i} |\bJ_{ij}|,
\end{equation}
and then say $\bJ$ is strictly diagonally dominant if $\Delta_i(\bJ) > 0 $ for $1 \leq i \leq k$.  We refer to a signal $\balpha$ as well-separated when its support set $\Omega_{\opt}$ satisfies this property.
\end{defn}

Note that once the minimum separation\footnote{Note that we measure distance in the DFT dictionary cyclically, so that column $d$ and column $1$ are $1$ column apart.} between any two non-zero elements of $\balpha$, denoted $h_{\min}$, is large enough, $\Omega_{\opt}$ is well-separated.  Indeed, this is easily due to the fact that the gram matrix of the dictionary $\bD$ has quickly decaying off-diagonal terms (see Fig. \ref{gram}).   We can quantify this notion by the following definition.

\begin{center}
\begin{figure}
\begin{center}
\includegraphics[height=1.8in]{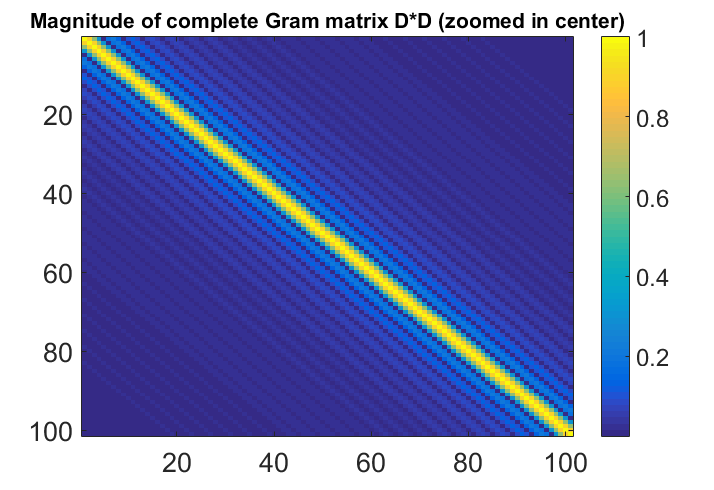}\quad\includegraphics[height=1.8in]{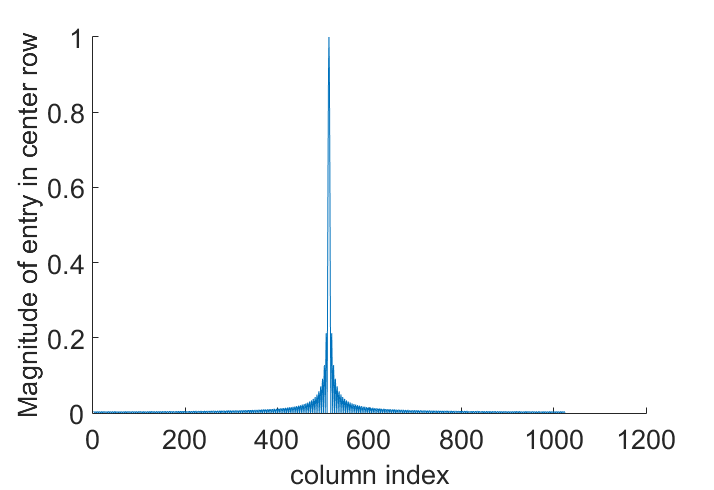}
\caption{Left: Magnitudes at center of gram matrix $\bD^*\bD$.  Right: Magnitudes in center row.}\label{gram}
\end{center}
\end{figure}
\end{center}

\added{
\begin{defn}[Dominance factor]
For an $n\times d$ DFT dictionary $\bD$, let $\cS_{k, h_{\min}}$ denote the set of all support sets $\Omega$ of size $k$ that have a minimum separation of at least $h_{\min}$:
$$
\cS_{k, h_{\min}} = \left\{\Omega : |\Omega| = k, \max_{p,q\in\Omega, p>q}\min(p-q, d-p+q) \geq h_{\min}\right\}.
$$
We define the \textnormal{dominance factor} $\eta_{h_{\min}, k}$ as the worst possible sum of off diagonals $\sum_{q \neq p} |\bG_{pq}|$ for $\bG = \bD_{\Omega}^*\bD_{\Omega}$ when $\Omega$ has $k$ entries and minimum separation $h_{\min}$.  That is,
\begin{equation}\label{defdom}
\eta_{h_{\min}, k} := \max_{\Omega\in\cS_{k, h_{\min}}}\max_{p\in\Omega} \sum_{q\in\Omega, q \neq p} \left|(\bD^*\bD)_{pq} \right|.
\end{equation}
We can also use this quantity to bound the correlations between atoms in a support set and an outside atom.  We may define
\begin{equation}\label{defdom2}
\eta_{h_{\min}, k}' := \max_{\Omega\in\cS_{k, h_{\min}}}\max_{p\notin\Omega} \sum_{q\in\Omega} \left|(\bD^*\bD)_{pq} \right|.
\end{equation}
\end{defn} 
}

With this notation, one observes that for a fixed sparsity level $k$, if $\eta_{h_{\min}, k} < (\bD^*\bD)_{qq}$ (where with our normalization $(\bD^*\bD)_{qq} = 1$) then any support set $\Omega_{\opt}$ of size $k$ and minimum separation $h_{\min}$ is guaranteed to satisfy the well-separated condition.  In this sense, utilizing these quantities will yield ``worst-case'' bounds, where we fix $h_{\min}$ and provide bounds for the worst possible support set.  Our bounds will thus be highly pessimistic but can be presented in a more simple form, depending only on $h_{\min}$.  We also remark that $\eta_{h_{\min}, k} \leq \eta_{h_{\min}, k}'$ but that $\eta_{h_{\min}, k} \approx \eta_{h_{\min}, k}'$ as the redundancy in the dictionary grows large.  Nonetheless, we treat these two quantities separately since for some dictionaries they may be significantly different.  

Figure \ref{diag_plot} shows numerically how separated the support of the signal needs to be to satisfy the well-separated condition\footnote{To obtain a crude upper bound on $\eta_{h_{\min},k}$ we simply assume that $\Omega_{\opt}$ contains equally separated non-zero elements, a distance $h_{\min}$ apart and sum the coherence between those atoms and the worst off-support column.  For that reason, when this bound is less than one, we guarantee that {any} support set with separation $h_{\min}$ is well-separated.}.  This again confirms that as long as the minimum separation $h_{\min}$ is large enough, the well-separated property holds.  Of course note that this computation verifies the minimum separation required for the \textit{worst case} support with that separation, whereas there may be many other support sets with smaller minimum separation which still guarantee the well-separated property.  Nonetheless, we will continue to utilize this quantity to bound the performance of OMP so that we may state the results only in terms of the separation $h_{\min}$.  

We also plot a (crude) upper bound\footnote{To bound $\eta_{h_{\min}, k}'$, we first upper bound the coherence between two columns separated by a distance $h$ by $f(h) := 1/n \cdot csc|h\pi/d|$.  Then one observes that because $f(h)$ is convex, the maximum in \eqref{defdom2} is attained when column $p$ is a neighbor of an element in $\Omega_{\opt}$. Hence $\eta_{h_{\min}, k}'$ is bounded by the sum of the coherence between columns in $\Omega_{\opt}$ with respect to such a column p. The bound can then be written formally as $f(1) + \sum_{j = 1}^{r} f( jh+1 ) + f( jh-1)$  where $r = \lfloor(k+1)/2\rfloor$.} on $\eta_{h_{\min}, k}'$ in Figure \ref{diag_plot}.  We will see later that we also want this quantity to be smaller than one.  

\begin{figure}
\includegraphics[width=3in]{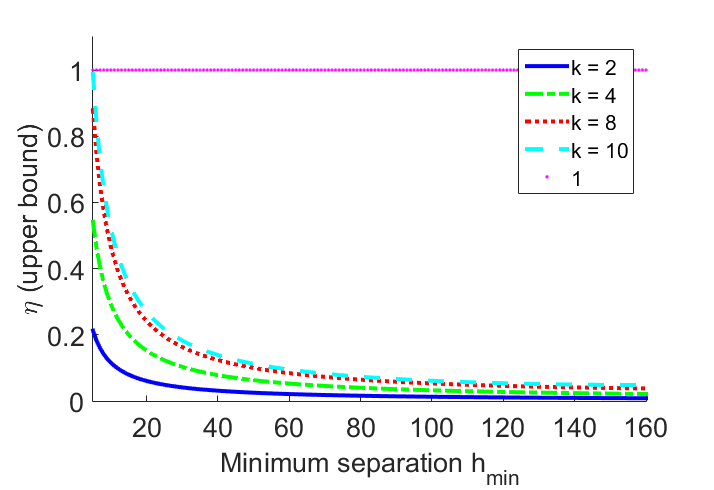}\quad \includegraphics[width=3in]{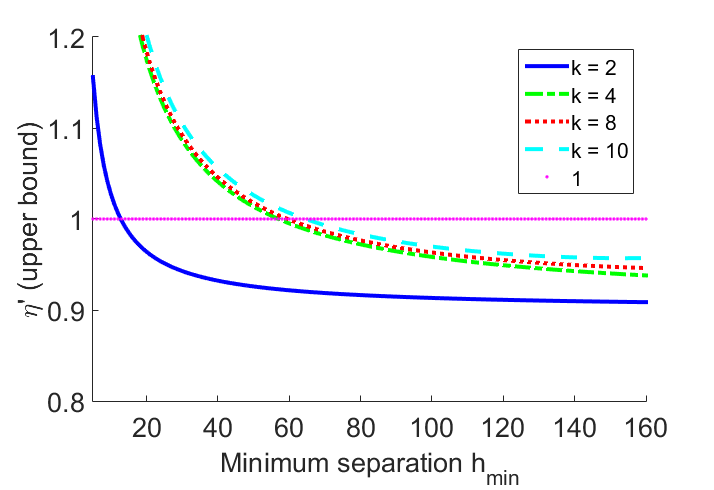}
\caption{Dictionary $\bD$ is $256\times 1024$ overcomplete DFT. Left: upper bound on the dominance factor $\eta_{h_{\min}, k}$ for support size $k$ and minimum separation $h_{\min}$.  Horizontal line shows value of diagonal entry. When curve falls below horizontal line the well-separated property is guaranteed to hold.  Right: upper bound on $\eta_{h_{\min}, k}'$ as a function of $h_{\min}$.}\label{diag_plot}
\end{figure}

Finally, we will utilize the following lemma, which bounds the restricted isometry constant when the signal has a well-separated support $\Omega_{\opt}$.

\begin{lem}
\label{D_RIP}
$\bD$ satisfies the following variant of the Restricted Isometry Property \eqref{RIP} with parameters ($k, \delta_k$):
\begin{equation}\label{ripbound}
(1-\delta_k)\|\balpha\|_2^2 \leq \|\bD\balpha\|_2^2 \leq (1+\delta_k)\|\balpha\|_2^2
\end{equation}
 for any well-separated $k$-sparse vector $\balpha$. In particular, if the minimum separation distance between the nonzeros of $\balpha$ is denoted by $h_{\min}$, then 

\begin{equation} \label{eq:e6}
\delta_k \leq \eta_{h_{\min}, k},
\end{equation}
where $\eta_{h_{\min}, k}$ is defined in \eqref{defdom}.
\end{lem}

\begin{proof}
Fix a signal $\balpha$ with well-separated support $\Omega = \Omega_{\opt}$, set $ \bG := \bD_{\Omega}^{*} \bD_{\Omega}$ and let $g_{pq}$ be the entries of $\bG$. According to the Gershgorin circle theorem \cite{gershgorin1931uber}, for every singular value $\lambda_i$ of $\bG$, we have 

\begin{equation} \label{eq:p1}
1-R_{p} \leq \lambda_i \leq 1+R_{p}
\end{equation}

\noindent for some $p\in\Omega$, where $R_{p} = \sum_{q\neq p} |g_{pq}| $ is the radius of the Gershgorin disc.
By definition, 

\begin{equation} \label{eq:p2}
\|\bG - \Id\|_{2} = \max_i(\lambda_i -1) \leq \max_{p \in \Omega} R_{p} = \max_{p \in \Omega} \sum_{q \in \Omega, q \neq p} |g_{pq}|.
\end{equation}

Since we know the minimum distance (with respect to the full dictionary $\bD$) between any two columns in $\bD_{\Omega_{\opt}}$ is $h_{\min}$, then we have by the definition \eqref{defdom} that 

\begin{equation} \label{eq:p3}
\|\bG - \Id\|_{2} = \max_{p \in \Omega} \sum_{q \in \Omega, q \neq p} |g_{pq}| \leq \eta_{h_{\min}, k}. 
\end{equation}

Lastly, the following observation proves the claim: 

\begin{equation}
\|\bG - \Id\|_{2}  = \sup_{\|\balpha\|_2 = 1} |\langle (\bD_{\Omega_{opt}}^*\bD_{\Omega_{opt}}  - \bI d)\, \balpha, \balpha \rangle| = \sup_{\|\balpha\|_2 = 1}\left|\, \|\bD \balpha\|_2^2 - \|\balpha\|_2^2\right|.
\end{equation}

\end{proof}

\subsubsection{Guarantees for SSCoSaMP(OMP)}

We consider here the case when the measurements are corrupted with noise, allowing the noiseless case to follow as a special instance when the noise has norm zero.  If our signal is contaminated by noise, then \eqref{eq:e2} becomes 

\begin{equation} \label{eq:n1}
\bx = \sum_{i\in\Omega_{opt}} \balpha_{i} \bvarphi_{i} + \be, ~~ \|\be\|_{2} \leq \varepsilon.
\end{equation}

We are now prepared to state our main result, which guarantees accurate recovery of well-separated signals under modest amounts of noise.

\begin{thm}
\label{OMP_withnoise}

Suppose that $\|\be\|_{2} < \varepsilon$, and 

\begin{equation} \label{eq:corollary_assumption}
B(h_{\min}):= \frac{\eta'_{h_{\min},k}}{1 - \eta_{h_{\min},k}} < 1,
\end{equation}
where both terms are defined in \eqref{defdom} and \eqref{defdom2}.
\noindent Then OMP (with the stopping rule $\|\br_i\|_2 \leq \varepsilon$) will exactly recover $\Omega_{\opt}$ if the minimum magnitude of the nonzero elements of the signal $\balpha$ satisfies

\begin{equation} \label{eq:n3}
\min_{i \in \Omega_{opt}} |\balpha_i|  \geq \frac{2\varepsilon}{1-\eta_{h_{\min},k}-\eta'_{h_{\min},k}}.
\end{equation}
\end{thm}

\begin{proof}
The proof of this result follows in the same fashion as the proof of \cite[Prop. 1]{OMPnoise}.  Denote the set of all selected indices in the first $t$ iterations by $\Omega_t$, and by ${\Omega_{t}^{\star}} := {\Omega_{\opt}} \backslash {\Omega_t}$ the remaining coefficients of the support yet to be identified.  We maintain the notation that the dictionary $\bD$ has columns $\bphi_i$ for $i=1, 2, \ldots d$.

  We proceed by induction on $t$, and suppose that $\bD_{\Omega_{t}} \subset \bD_{\Omega_{\opt}}$ after $t$ iterations. Define $\bP_t := \bD_{\Omega_t}\bD_{\Omega_t}^{\dagger}$. Then the residual after $t$ steps is

\begin{equation}
\br_t = (\bI - \bP_t)\by = (\bI - \bP_t)\bD \balpha + (\bI - \bP_t) \be =: \bs_t + \bn_t
\end{equation}

\noindent where $\bs_t := (\bI - \bP_t)\bD \balpha$ and $\bn_t := (\bI - \bP_t) \be$. Let

\begin{equation}
M_{t} = \max_{i \in {\Omega_t^{\star}}} \left|\bphi_i^*\bs_t\right|, \quad M_{t}' = \max_{i \in {\Omega^c_{\opt}}} \left|\bphi_i^*\bs_t\right|, \quad N_t = \max_{i} \left|\bphi_i^* \bn_t\right|.
\end{equation}

In order for OMP to select a correct index at the next iteration, we need $\max_{i \in {\Omega_t^{\star}}} \left|\bphi_i^*\br_t\right| > \max_{i \in {\Omega^c_{\opt}}} \left|\bphi_i^*\br_t\right|$.  It suffices to require that $M_{t} - M_{t}'>2N_t$ since that would imply

\begin{equation}
\max_{i \in {\Omega_t^{\star}}} \left|\bphi_i^*\br_t\right| \geq M_{t} - N_t > M_{t}' + N_t \geq \max_{i \in {\Omega^c_{\opt}}} \left|\bphi_i^*\br_t\right|.
\end{equation}

We utilize the following lemma.

\begin{lem} \cite[Proof of Theorem 4.2]{Reading5}
In our notation, $B(h_{\min}) M_{t} > M_{t}'$ for all $t$.
\end{lem}

By this lemma, $M_{t} - M_{t}' > (1-B(h_{\min})) M_{t}$. Hence it suffices to require $M_{t} > \frac{2}{1-B(h_{\min})} N_t$ for OMP to select a correct index.  Then we have by definition of $\Omega_t^{\star}$, 

\begin{equation}
M_{t} = \|\bD_{\Omega_t^{\star}}^* \bs_t\|_{\infty} = \|\bD_{\Omega_t^{\star}}^* (\bI - \bP_t) \bD \balpha\|_{\infty} = \|\bD_{\Omega_t^{\star}}^* (\bI - \bP_t) \bD_{\Omega_t^{\star}} \balpha_{\Omega_t^{\star}}\|_{\infty}.
\end{equation}

The following lemma allows us to relate the above expression to $\bD_{\Omega_{\opt}}^*\bD_{\Omega_{\opt}}$. 

\begin{lem} \label{eiglem} \cite[Lemma 5]{OMPnoise}
In our notation, $\lambda_{\min}\left(\bD_{\Omega_{\opt}}^*\bD_{\Omega_{\opt}}\right) \leq \lambda_{\min}\left(\bD_{\Omega_t^{\star}}^* (\bI - \bP_t) \bD_{\Omega_t^{\star}}\right)$.
\end{lem}

Thus we can bound, 

\begin{equation}
\|\bD_{\Omega_t^{\star}}^* (\bI - \bP_t) \bD_{\Omega_t^{\star}} \balpha_{\Omega_t^{\star}}\|_2 \geq \lambda_{\min}(\bD_{\Omega_{\opt}}^*\bD_{\Omega_{\opt}}) \|\balpha_{\Omega_{t}^{\star}} \|_2.
\end{equation}

Combining this with the fact that $|\Omega_t^{\star}| = k-t$, we have

\begin{equation}
M_{t} \geq \frac{1}{\sqrt{k-t}} \|\bD_{\Omega_t^{\star}}^* \bs_t\|_2 = \frac{1}{\sqrt{k-t}}\|\bD_{\Omega_t^{\star}}^* (\bI - \bP_t) \bD_{\Omega_t^{\star}} \balpha_{\Omega_t^{\star}}\|_2 \geq \frac{\lambda_{\min}(\bD_{\Omega_{\opt}}^*\bD_{\Omega_{\opt}})}{\sqrt{k-t}} \|\balpha_{\Omega_{t}^{\star}} \|_2.
\end{equation}

Therefore, for the sufficient condition $M_{t} - M_{t}'>2N_t$ to hold, it is enough that

\begin{equation}
\|\balpha_{\Omega_{t}^{\star}} \|_2 > \frac{2 \sqrt{k-t} N_t}{(1-B(h_{\min})) \lambda_{\min}(\bD_{\Omega_{\opt}}^*\bD_{\Omega_{\opt}})}.
\end{equation}

Since $\|\be\|_2 \leq \varepsilon$, 

\begin{equation}
\|\bn_t\|_2 = \|(\bI - \bP_t) \be \|_2 \leq \|\be\|_2 \leq \varepsilon.
\end{equation}

Thus for any column $\bphi_i$ of $\bD$, 

\begin{equation}
|\bphi_i^* \bn_t | \leq \|\bphi_i\|_2 \|\bn_t\|_2 \leq \varepsilon,
\end{equation}
which implies that $N_t \leq \varepsilon$.  
Since $\lambda_{\min}(\bD_{\Omega_{\opt}}^*\bD_{\Omega_{\opt}}) \geq 1 - \eta_{h_{\min},k}$ by Lemma~\ref{D_RIP}, we have that

\begin{equation}\label{alph}
\|\balpha_{\Omega_{t}^{\star}} \|_2 \geq \frac{2\sqrt{k-t} \varepsilon}{(1-B(h_{min}))(1-\eta_{h_{\min},k})} = \frac{2\sqrt{k-t} \varepsilon}{(1-\eta_{h_{\min},k}-\eta'_{h_{\min},k})}
\end{equation}
suffices to ensure that a correct index will be selected at this step.  Observing that $|\Omega_{t}^{\star}| = k-t$ yields the desired condition \eqref{eq:n3}.

It remains to argue that the stopping criterion is sufficient for termination.  It suffices to show that at each step OMP doesn't stop early -- that with this stopping criterion it runs the full $k$ iterations. Equivalently, we want to show that for $t<k$, $\|\br_t\|_2 > \varepsilon$.  We observe that

\begin{equation}
\begin{split}
\|\br_t\|_2 &= \|(\bI - \bP_t)\bD \balpha + (\bI - \bP_t)\be \|_2 \\
&\geq \|(\bI - \bP_t)\bD \balpha\|_2 - \|(\bI - \bP_t) \be\|_2\\
&\geq \|(\bI - \bP_t)\bD_{\Omega_{t}^{\star}} \balpha_{\Omega_{t}^{\star}}\|_2 - \varepsilon.
\end{split}
\end{equation}

Again using Lemma \ref{eiglem} along with \eqref{alph}, we have for $t \leq k$, 

\begin{align*}
\|(\bI - \bP_t)\bD_{\Omega_{t}^{\star}} \balpha_{\Omega_{t}^{\star}}\|_2 &\geq \lambda_{\min}(\bD_{\Omega_{\opt}}^*\bD_{\Omega_{\opt}}) \|\balpha_{\Omega_{t}^{\star}}\|_2 \\
&\geq \lambda_{\min}(\bD_{\Omega_{\opt}}^*\bD_{\Omega_{\opt}}) \frac{2 \sqrt{k-t}\varepsilon}{(1-B(h_{\min})) \lambda_{\min}(\bD_{\Omega_{\opt}}^*\bD_{\Omega_{\opt}})} \\
&= \frac{2 \sqrt{k-t}\varepsilon}{(1-B(h_{\min})) }
> 2 \varepsilon.
\end{align*}

Thus, until $|\Omega_{t}^{\star}| = k - t = 0$ (i.e. until we have identified all of the support), as desired we have

\begin{equation}
\|\br_t\|_2 > 2 \varepsilon - \varepsilon = \varepsilon.
\end{equation}

\end{proof} 

We again numerically display the bounds required by Theorem \ref{OMP_withnoise}.  For dimensions $n=256,~d=1024$, noise level $\varepsilon = 10^{-3}$, and again using the same bounds on $\eta_{h_{\min},k}$ and $\eta_{h_{\min},k}'$ as described above, we empirically show several bounds with increasing sparsity levels in Figure~\ref{fig:noise}. 
\begin{figure}[ht]
\centering
{{\includegraphics[scale=.4]{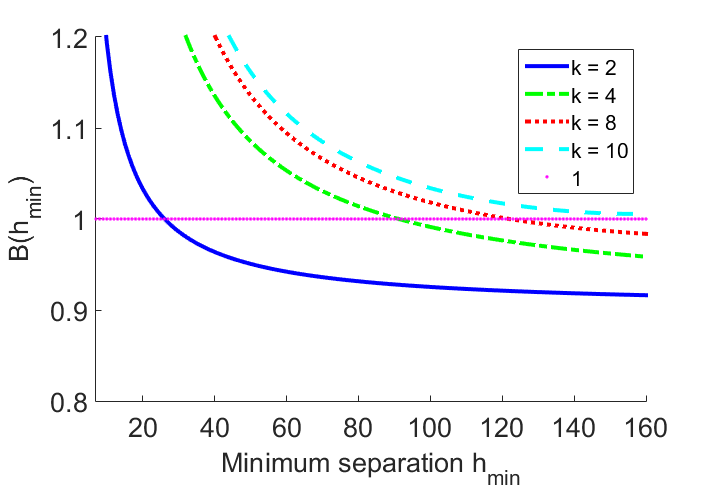}} } \quad {{\includegraphics[scale=.4]{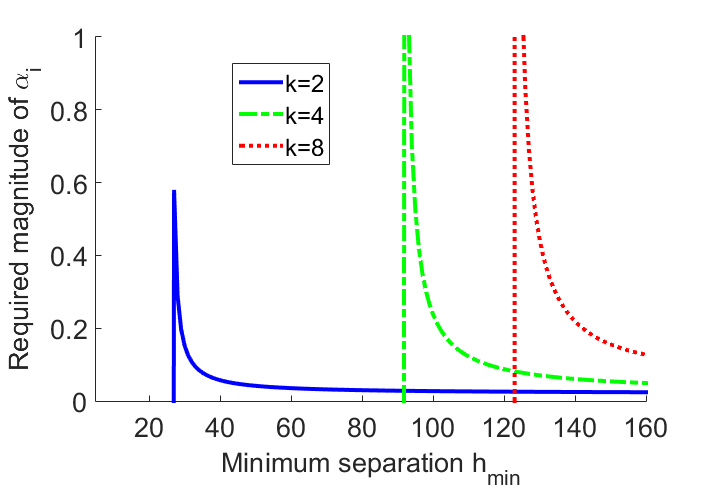}} }
\caption{Dictionary is $256\times 1024$ DFT.  Left: Plot of upper bound on $B(h_{\min})$ vs minimum separation for various sparsity levels.  When the curve falls below $1$ (dotted line), Theorem \ref{OMP_withnoise} guarantees that OMP provides accurate recovery.  Right: Lower bound on non-zeros in signal $\balpha$ as in Theorem \ref{OMP_withnoise} when noise is bounded by $10^{-3}$ as a function of minimum separation $h_{\min}$.\label{fig:noise}}

\end{figure}

\subsection{{SSCoSaMP(L1)}}

We conclude with a simple argument showing that the $\ell_1$-minimization method can provide a set of good near-optimal projection components $\cS_D (\bx,k)$ even if $\bD$ is an overcomplete dictionary (partial DFT again for simplicity) under the condition that the signal is well-separated.  
We will utilize an immediate corollary of \cite{Reading2}.

\begin{cor} \label{cor1}  \cite{Reading2}
Consider the program
\begin{equation} \label{eq:eq16}
\argmin_{\tilde{\bx}} \| \tilde{\bx} \|_1 \text{   subject to    } \bF_n \tilde{\bx} = \bF_n {\bx},
\end{equation}
 where $\bF_n$ is the partial Fourier matrix in a low-pass band $[-f_{lo},f_{lo}]$.
Then if T $\subset$ $\{0,1,...,N-1\}$ is the support of $\{\bx_t\}^{N-1}_{t = 0}$ obeying

\begin{equation} \label{eq:eq15}
\min_{t,t' \in T: t \not = t'} \frac{1}{N} |t - t'| \geq  2/f_{lo},
\end{equation}
the solution to \eqref{eq:eq16} is exact, $\tilde{\bx} = \bx$.

\end{cor}

In our context, the overcomplete dictionary $\bD$ can only collect the lowest $2 f_{lo} +1$ frequencies. 
When the dictionary $\bD$ has dimension $n \times d$ with $n<d$, $\bD\bx$ only gives the spectrum of $\bx$ in the frequency range $[-\frac{n}{2}, \frac{n}{2}]$. Here then, $\frac{n}{2} = f_{lo}$. Therefore $\bD$ acts as a low-pass filter in our problem and we seek to recover $\balpha$ by only knowing parts of the spectrum of $\balpha$.  We can thus translate the above result to the context of SSCoSaMP projections.

\begin{cor} 
Let $\bD$ be an $n \times d$ ($n \leq d$) overcomplete DFT dictionary that gives the Fourier transform of a vector in the frequency domain $[-f_{lo},f_{lo}] := [-n/2, n/2]$. Suppose $\bx$ has a $k$-sparse expansion in $\bD$, i.e. $\bx = \bD \balpha$ where $\balpha$ has support $T$ with $|T|\leq k$.  If the support obeys  

\begin{equation}\label{eq:eq16.3}
\min_{t,t' \in T: t \not = t'} \frac{1}{d} |t - t'| \geq 2/f_{lo}
\end{equation}

\noindent then the solution to 

\begin{equation}\label{eq:eq16.7}
\argmin_{\hat{\balpha}} \|\hat{\balpha}\|_1 \text{    subject to    }  \bD \hat{\balpha} = \bx
\end{equation}

\noindent is exact ($\hat{\balpha} = \balpha$) and can thus be used as a projection within SSCoSaMP.

\end{cor}

This result provides theoretical backing to the behavior observed by SSCoSaMP($\ell_1$) in Figure \ref{sepvclust}.

\section{Conclusion}\label{sec:conc}

In this work we study the behavior of several classical CS methods as approximate projections within a greedy signal space method.  It has been observed that the accuracy of the methods in this setting depends heavily on the signal structure, even though their classical recovery guarantees are independent of signal structure.  Here, we analyze the behavior of these approaches when the dictionary may be highly overcomplete and thus does not satisfy typical properties like the RIP or incoherence.  We prove that two of the methods, under specific assumptions on the signal structure,  can be used as accurate approximate projections.  Our analysis thus provides theoretical backing to explain the observed phenomena.  It would be useful future work to study additional methods, like the CoSaMP method which we conjecture can serve as an approximate projection when the signal has a clustered support.


\bibliographystyle{myalpha}

\bibliography{../../../../bib}

\newcommand{\etalchar}[1]{$^{#1}$}
\begin{thebibliography}{GGK{\etalchar{+}}14}

\bibitem[BD09]{PaperIHT}
T.~Blumensath and M.~E. Davies.
\newblock Iterative hard thresholding for compressed sensing.
\newblock {\em Appl. Comput. Harmon. A.}, 27(3):265--274, 2009.

\bibitem[CENR10]{RefWorks:60}
E.~J. Cand\`es, Y.~C. Eldar, D.~Needell, and P.~Randall.
\newblock Compressed sensing with coherent and redundant dictionaries.
\newblock {\em Appl. Comput. Harmon. A.}, 31(1):59--73, 2010.

\bibitem[CFG14]{Reading2}
E.~J. Cand{\`e}s and C.~Fernandez-Granda.
\newblock Towards a mathematical theory of super-resolution.
\newblock {\em Commun. Pur. Appl. Math.}, 67(6):906--956, 2014.

\bibitem[CRT06]{RefWorks:47}
E.~J. Cand\`es, J.~Romberg, and T.~Tao.
\newblock Stable signal recovery from incomplete and inaccurate measurements.
\newblock {\em Commun. Pur. Appl. Math.}, 59(8):1207--1223, 2006.

\bibitem[CRVT05]{RefWorks:444}
E.~J. Cand\`es, M.~Rudelson, R.~Vershynin, and T.~Tao.
\newblock Error correction via linear programming.
\newblock {\em 46th Annual Symp. Found. Computer Science}, pages 668--681,
  2005.

\bibitem[CT05]{RefWorks:48}
E.~J. Cand\`es and T.~Tao.
\newblock Decoding by linear programming.
\newblock {\em IEEE T. Inform. Theory}, 51:4203--4215, 2005.

\bibitem[CW11]{OMPnoise}
T.~Cai and L.~Wang.
\newblock Orthogonal matching pursuit for sparse signal recovery with noise.
\newblock {\em Information Theory, IEEE Transactions}, 57:4680--4688, 2011.

\bibitem[DNW12]{Paper5}
M.~Davenport, D.~Needell, and M.~B. Wakin.
\newblock Signal space {CoSaMP} for sparse recovery with redundant
  dictionaries.
\newblock {\em IEEE T. Inform. Theory}, 59(10):6820, 2012.

\bibitem[DW12]{RefWorks:23}
M.~A. Davenport and M.~B. Wakin.
\newblock Compressive sensing of analog signals using discrete prolate
  spheroidal sequences.
\newblock {\em Appl. Comput. Harmon. A.}, 33(3):438--472, 2012.

\bibitem[Fou]{foucart15}
S.~Foucart.
\newblock Dictionary-sparse recovery via thresholding-based algorithms.
\newblock {\em J. Fourier Anal. Appl.}
\newblock To appear.

\bibitem[GE13]{giryes2013omp}
R.~Giryes and M.~Elad.
\newblock Omp with highly coherent dictionaries.
\newblock In {\em 10th Int. Conf. on Sampling Theory Appl.(SAMPTA)}, 2013.

\bibitem[Ger31]{gershgorin1931uber}
S.~A. Gershgorin.
\newblock Uber die abgrenzung der eigenwerte einer matrix.
\newblock {\em Izv. Akad. Nauk. USSR Otd. Fiz.-Mat.}, (6):749--754, 1931.

\bibitem[GGK{\etalchar{+}}14]{practicalGGKLNT14}
C.~Garnatz, X.~Gu, A.~Kingman, J.~LaManna, D.~Needell, and S.~Tu.
\newblock Practical approximate projection schemes in greedy signal space
  methods.
\newblock In {\em Proc. Allerton Conf. on Communication, Control, and
  Computing}, 2014.

\bibitem[GN15]{giryes2013greedy}
R.~Giryes and D.~Needell.
\newblock Greedy signal space methods for incoherence and beyond.
\newblock {\em Applied and Computational Harmonic Analysis}, 39(1):1--20, 2015.

\bibitem[GNE{\etalchar{+}}14]{RefWorks:607}
R.~Giryes, S.~Nam, M.~Elad, R.~Gribonval, and M.~E. Davies.
\newblock Greedy-like algorithms for the cosparse analysis model.
\newblock {\em Linear Algebra Appl.}, 441:22--60, 2014.

\bibitem[Mut05]{RefWorks:279}
S.~Muthukrishnan.
\newblock Data streams: Algorithms and applications.
\newblock 2005.

\bibitem[NDEG13]{RefWorks:271}
S.~Nam, M.~E. Davies, M.~Elad, and R.~Gribonval.
\newblock The cosparse analysis model and algorithms.
\newblock {\em Applied and Computational Harmonic Analysis}, 34(1):30--56,
  2013.

\bibitem[NT09]{NeedeT_CoSaMP}
D.~Needell and J.~Tropp.
\newblock {CoSaMP}: {I}terative signal recovery from incomplete and inaccurate
  samples.
\newblock {\em Appl. Comput. Harmon. A.}, 26(3):301--321, 2009.

\bibitem[NV07]{RefWorks:44}
D.~Needell and R.~Vershynin.
\newblock Signal recovery from incomplete and inaccurate measurements via
  regularized orthogonal matching pursuit.
\newblock {\em IEEE J. Sel. Top. Signa.}, 4:310--316, 2007.

\bibitem[RRT12]{rauhut2012restricted}
H.~Rauhut, J.~Romberg, and J.~A. Tropp.
\newblock Restricted isometries for partial random circulant matrices.
\newblock {\em Appl. Comput. Harmon. A.}, 32(2):242--254, 2012.

\bibitem[RV08]{RefWorks:285}
M.~Rudelson and R.~Vershynin.
\newblock On sparse reconstruction from fourier and gaussian measurements.
\newblock {\em Comm. Pure Appl. Math.}, 61:1025--1045, 2008.

\bibitem[TG07]{Paper9}
J.~A. Tropp and A.~C. Gilbert.
\newblock Signal recovery from random measurements via {O}rthogonal {M}atching
  {P}ursuit.
\newblock {\em IEEE T. Inform. Theory}, 53(12):4655--4666, 2007.

\bibitem[Tro04]{Reading5}
J.~A. Tropp.
\newblock Greed is good: Algorithmic results for sparse approximation.
\newblock {\em IEEE T. Inform.Theory}, 50(10):2231--2242, Oct. 2004.

\bibitem[Zha11]{RefWorks:150}
T.~Zhang.
\newblock Sparse recovery with orthogonal matching pursuit under {RIP}.
\newblock {\em IEEE T. Inform.Theory,}, 57(9):6215--6221, 2011.

\end{thebibliography}

\end{document}